\documentclass[11pt]{amsart}

\usepackage{amscd}

\usepackage{amsthm}

\newtheorem{thm}[equation]{Theorem}
\newtheorem{lem}[equation]{Lemma}
\newtheorem{cor}[equation]{Corollary}
\newtheorem{prop}[equation]{Proposition}

\newtheorem*{thm*}{Theorem}
\newtheorem*{prop*}{Proposition}
\newtheorem*{cor*}{Corollary}
\newtheorem*{lem*}{Lemma}
\newtheorem*{MT*}{Main Theorem}

\theoremstyle{definition} %
\newtheorem{defn}[equation]{Definition}
\newtheorem*{defn*}{Definition}

\newtheorem{eg}[equation]{Example}

\theoremstyle{remark} %

\newtheorem*{rmk*}{Remark}
\newtheorem*{rmks*}{Remarks}

\newtheoremstyle{exercise}
  {3pt}
  {3pt}
  {\small}
  {}
  {\sc\small}
  {.}
  {.5em}
   {}     
  {}

\theoremstyle{exercise}

\makeatletter
{\renewcommand{\theequation}{#1}}%
{\renewcommand{\theequation}{\arabic{equation}}\addtocounter{equation}{-1}\global\@ignoretrue}
\makeatother

\makeatletter
{\renewcommand{\theequation}{#1}\begin{eqnarray}}%
{\end{eqnarray}\renewcommand{\theequation}{\arabic{equation}}\addtocounter{equation}{-1}\global\@ignoretrue}
\makeatother

\makeatletter
{\smallskip \refstepcounter{equation}\noindent{\textbf{\theequation.} }{{\textbf{#1.}}}}%
{\smallskip \global\@ignoretrue}
\makeatother

\makeatletter
{\smallskip \refstepcounter{equation}\noindent{\textbf{\theequation.} }{{\textbf{#1}}}}%
{\smallskip \global\@ignoretrue}
\makeatother

\makeatletter
{\smallskip \refstepcounter{equation}{\sc \theequation}{\sc (#1).}}%
{\smallskip \global\@ignoretrue}
\makeatother

\makeatletter
{\smallskip \refstepcounter{equation}\noindent{\textbf\theequation.}{\textsl{ #1.}}}%
{\smallskip \global\@ignoretrue}
\makeatother

\makeatletter
\newenvironment{borel*}%
{\smallskip \refstepcounter{equation}\noindent{\textbf{\theequation.}}}%
{\global\@ignoretrue}
\makeatother

\newcommand{\flist}[1]{\hangindent\leftmargini\textup{(1)}\hskip\labelsep {#1}%
\begin{enumerate}%
\setcounter{enumi}{1}%
}

\newcommand{\ot}{\otimes}

\newcommand{\eand}{\quad\text{and}\quad}

\newcommand{\Q}{{\mathbb{Q}}}        
\newcommand{\F}{{\mathbb{F}}}        
\newcommand{\Z}{{\mathbb{Z}}}        

\newcommand{\Zm}[1]{\Z/{#1}\Z}

\newcommand{\oddots}{{\mathinner{\mkern1mu\raise1pt\vbox{\kern7pt\hbox{.}}\mkern2mu\raise4pt\hbox{.}\mkern2mu\raise7pt\hbox{.}\mkern1mu}}}

\newcommand{\Fx}{{F^{\times}}}

\newcommand{\qform}[1]{{\left\langle{#1}\right\rangle}}                   
\newcommand{\pform}[1]{{\langle\!\langle{#1}\rangle\!\rangle}} 

\DeclareMathOperator{\res}{res}
\DeclareMathOperator{\cores}{cor}

\newcommand{\Nrd}{{\mathrm{Nrd}}}

\newcommand{\iso}{\xrightarrow{\sim}}

\newcommand{\ra}{\rightarrow}

\numberwithin{equation}{section}

\newcommand{\ba}{\bar{a}}
\newcommand{\bb}{\bar{b}}

\newcommand{\Lb}{\bar{L}}
\newcommand{\Eb}{\bar{E}}
\newcommand{\Fb}{{\bar{F}}}
\newcommand{\Bb}{\bar{B}}
\newcommand{\Ab}{\bar{A}}

\newcommand{\Fh}{{\hat{F}}}
\newcommand{\Lh}{{\hat{L}}}
\newcommand{\Ah}{\hat{A}}
\newcommand{\ah}{{\hat{\alpha}}}

\newcommand{\ram}{r}

\begin{document}

\title{Quaternion algebras with the same subfields}

\author{Skip Garibaldi}
\address{(Garibaldi) Department of Mathematics \& Computer Science, Emory University, Atlanta, GA 30322, USA}
\email{skip@member.ams.org}
\urladdr{http://www.mathcs.emory.edu/{\textasciitilde}skip/}

\author{David J. Saltman}
\address{(Saltman) Center for Communications Research-Princeton, 805 Bunn Drive, Princeton, NJ 008540, USA}
\email{saltman@idaccr.org}




\begin{abstract}
G.~Prasad and A.~Rapinchuk asked if two quaternion division $F$-algebras that have the same subfields are necessarily isomorphic.  The answer is known to be ``no" for some very large fields.  We prove that the answer is ``yes" if $F$ is an extension of a global field $K$ so that $F/K$ is unirational and has zero unramified Brauer group.  We also prove a similar result for Pfister forms and give an application to tractable fields.
\end{abstract}
\maketitle

\section{Introduction}

Gopal Prasad and Andrei Rapinchuk asked the following question in Remark 5.4 of their paper \cite{PrRap:weakly}:
\begin{equation}\label{ques}
\parbox{4.25in}{\emph{If two quaternion division algebras over a field $F$ have the same maximal subfields, are the algebras necessarily isomorphic?}}
\end{equation}
The answer is ``no" for some fields $F$, see \S\ref{Rost.eg} below.
The answer is ``yes" if $F$ is a global field by the Albert-Brauer-Hasse-Minkowski Theorem \cite[8.1.17]{NSW2}.    Prasad and Rapinchuk note that the answer is unknown even for fields like $\Q(x)$.  We prove that the answer is ``yes" for this field:

\begin{thm} \label{MT}
Let $F$ be a field of characteristic $\ne 2$ that is transparent. If $D_1$ and $D_2$ are quaternion division algebras over $F$ that have the same maximal subfields, then $D_1$ and $D_2$ are isomorphic.
\end{thm}

The term ``transparent" is defined in \S\ref{transparent.sec} below.  Every retract rational extension of a local, global, real-closed, or algebraically closed field is transparent; in particular $K(x_1, \ldots, x_n)$ is transparent for every global field $K$ of characteristic $\ne 2$ and every $n$.  There are many other examples.

The theorem can be viewed as a statement about symbols in the Galois cohomology group $H^2(F, \Zm2)$.  We also state and prove an analogue for symbols in Galois cohomology of $H^d(F, \Zm2)$---i.e., $d$-Pfister quadratic forms---in Theorem \ref{pf.thm} below.


\subsection*{Notation}
A \emph{global field} is a finite extension of $\Q$ or of $\F_p(t)$ for some prime $p$.   A \emph{local field} is a completion of a global field with respect to a discrete valuation.

\section{An example} \label{Rost.eg}

\begin{eg}
Several people (in no particular order: Markus Rost, Kelly McKinnie, Adrian Wadsworth,  Murray Schacher, Daniel Goldstein...) noted that some hypothesis on the field $F$ is necessary for the conclusion of the theorem to hold.  Here is an example to illustrate this.  

Let $F_0$ be a field of characteristic $\ne 2$.  We follow \cite{Lam} for notation on quadratic forms, except that we define the symbol $\pform{a_1, \ldots, a_d}$ to be the $d$-Pfister form $\otimes_{i = 1}^d \qform{1, -a_i}$.  For a Pfister form $\phi$, we put $\phi'$ for the unique form such that $\phi \cong \qform{1} \oplus \phi'$.

Suppose that $F_0$ has quaternion division algebras $Q_1$, $Q_2$ that are not isomorphic.  Write $\phi_i := \pform{a_i, b}$ for the norm form of $Q_i$.  The Albert form $\phi_1 - \phi_2$ has anisotropic part similar to $\gamma = \pform{a_1 a_2, b}$.

Fix an extension $E/F_0$ and a proper quadratic extension $E(\sqrt{c})$ contained in $Q_i$ but not $Q_{i+1}$, with subscripts taken modulo 2.  Equivalently, $-c$ is represented by $\phi'_i$ but not by $\phi'_{i+1}$.  Put $q_{i,c} := \qform{c} \oplus \phi'_{i+1}$; it is anisotropic.  Further, its determinant $c$ is a nonsquare, so $q_{i,c}$ is not similar to $\phi_i$, $\phi_{i+1}$, nor $\gamma$.  All three of those forms remain anisotropic over the function field $E(q_{i,c})$ of $q_{i,c}$ by \cite[X.4.10(3)]{Lam}.  That is, $Q_1 \ot E(q_{i,c})$ and $Q_2 \ot E(q_{i,c})$ are division and contain $E(q_{i,c})(\sqrt{c})$, and the two algebras are still distinct.

We build a field $F$ from $F_0$ inductively.  Suppose we have constructed the field $F_j$ for some $j \ge 0$.  Let $F_{j+1}$ be the composita of the extensions $F_j(q_{i,c})$ for $c$ as in the previous paragraph with $E = F_j$ and $i = 1, 2$.  Let $F$ be the colimit of the $F_j$.

By construction, $Q_1 \ot F$ and $Q_2 \ot F$ are both division, are not isomorphic to each other, and have the same quadratic subfields.  This shows that some hypothesis on the field $F$  is necessary for the conclusion of the theorem to hold.
\end{eg}

One cannot simply omit the hypothesis ``division" in the theorem, by Example \ref{local.bad}.

Furthermore, Theorem \ref{MT} does not generalize to division algebras of degree $> 2$, even over number fields.  One can easily construct examples using the local-global principle for division algebras over a number field, see \cite[Example 6.5]{PrRap:weakly} for details. 

\section{Discrete valuations: good residue characteristic} \label{good.char}

Fix a field $F$ with a discrete valuation $v$.  We write $\Fb$ for the residue field and $\Fh$ for the completion.  (Throughout this note, for fields with a discrete valuation and elements of such a field, we use bars and hats to indicate the corresponding residue field/completion and residues/image over the completion.)

For a central simple $F$-algebra $A$, we write $\Ah$ for $A \ot \Fh$.  It is Brauer-equivalent to a central division algebra $B$ over $\Fh$.  As $v$ is complete on $\Fh$, it extends uniquely to a valuation on $B$ by setting $v(b) := \frac1{\deg B}\,v(\Nrd_B(b))$ for $b \in B^\times$, see, e.g., \cite[12.6]{Reiner}.  We say that $A$ is \emph{unramified} if $v(B^\times) = v(F^\times)$---e.g., if $\Fh$ splits $A$---and \emph{ramified} if $v(B^\times)$ is strictly larger.  That is, $A$ is ramified or unramified if $v$ is so on $B$.  Note that these meanings for ramified and unramified are quite a bit weaker than the the usual definitions (which usually require that the residue division algebra of $B$ be separable over $\Fb$).

The valuation ring in $\Fh$ is contained in a noncommutative discrete valuation ring $S$ in $B$, and the residue algebra of $S$ is a (possibly commutative) division algebra $\Bb$ whose center contains $\Fb$.  One has the formula
\[
\dim_\Fb \Bb \cdot [v(B^\times) : v(F^\times)] = \dim_\Fh B
\]
(see \cite[p.~393]{W:survey} for references).
That is, $A$ is unramified if and only if $\dim_{\Fb} \Bb = \dim_\Fh B$.

We use a few warm-up lemmas.

\begin{lem} \label{krasner}
Let $A_1, A_2$ be central simple algebras over a field $F$ that has a discrete valuation.  If there is a separable maximal commutative subalgebra that embeds in one of the $\Ah_i$ and not the other, then the same is true for the $A_i$.
\end{lem}

\begin{proof}
Let $\Lh$ be a separable maximal commutative subalgebra of $\Ah_1$ that is not a maximal subalgebra of $\Ah_2$.  Choose $\ah \in \Lh$ so that $\Lh = \Fh(\ah)$.  Since $A_1$ is dense in $\Ah_1$, we can choose $\alpha \in A_1$ as close as we want to $\ah$.  Thus we can choose the Cayley-Hamilton polynomial of $\alpha$ to be as close as we need to that of $\ah$.  In particular, by Krasner's Lemma (applied as in \cite[33.8]{Reiner}), we can assume that $\Fh(\alpha) \cong \Fh(\ah)$.  In particular, $\Fh(\alpha)$ is not a maximal subalgebra of $\Ah_2$, which implies that $F(\alpha)$ is not a maximal subalgebra of $A_2$ and is by construction a maximal subfield of $A_1$.  Since $\Fh(\alpha) \cong \Fh(\ah)$, the minimal polynomial of $\alpha$ over $F$ has nonzero discriminant and hence is separable.
\end{proof}

\begin{lem} \label{krasner.2}
Let $A_1, A_2$ be central division algebras of prime degree over a field $F$ with a discrete valuation.  If
\begin{enumerate}
\item $\Ah_1$ is split and $\Ah_2$ is not; or
\item both $\Ah_1$ and $\Ah_2$ are division, but only one is ramified,
\end{enumerate}
then there is a separable maximal subfield of $A_1$ or $A_2$ that does not embed in the other.
\end{lem}

\begin{proof}
The statement is obvious from Lemma \ref{krasner}.  In case (1), we take take the diagonal subalgebra in $\Ah_1$.  In case (2), we take a maximal separable subfield that is ramified.
\end{proof}
\smallskip

\emph{For the rest of this section,
we fix a prime $p$ and let $F$ be a field with a primitive $p$-th root of unity $\zeta$; we suppose that $F$ has a discrete valuation $v$ with residue field of characteristic different from $p$.}

A \emph{degree $p$ symbol algebra} is an (associative) $F$-algebra generated by elements $\alpha, \beta$ subject to the relations $\alpha^p = a$, $\beta^p = b$, and $\alpha \beta = \zeta \beta \alpha$.  It is central simple over $F$ \cite[p.~78]{Draxl} and we denote it by $(a, b)_F$ or simply $(a, b)$.

\begin{lem} \label{normalize.cyclic}
If $A$ is a degree $p$ symbol algebra over $F$, then $A$ is isomorphic to $(a, b)_F$ for some $a, b \in \Fx$ with $v(a) = 0$.
\end{lem}

\begin{proof}
Write $A$ as $(a,b)$ where $a = x \pi^{n_a}$ and $b = y \pi^{n_b}$ where $x, y$ have value 0 and $\pi$ is a uniformizer for $F$.  We may assume that $p$ divides neither $n_a$ nor $n_b$, so $n_a = sn_b$ modulo $p$.  But $A$ is isomorphic to $(a(-b)^{-s}, b)$ because $(-b, b)$ is split \cite[p.~82, Cor.~5]{Draxl}, and $v(a(-b)^{-s})$ is divisible by $p$.
\end{proof}

\begin{borel*}
Let $(a, b)$ be a symbol algebra as in Lemma \ref{normalize.cyclic} and write $\Fh$ for the completion of $F$ at $v$.  If $(a, b) \ot \Fh$ is not division, certainly $(a, b)$ is \emph{unramified}.  Otherwise, there are two cases:
\begin{itemize}
\item If $v(b)$ is divisible by $p$, then we can assume $v(b) = 0$, $(a, b)$ is unramified, and $\overline{(a,b)}$ is the symbol algebra $(\ba, \bb)_{\Fb}$.  (The residue algebra---which is division---contains $(\ba,\bb)_{\Fb}$, hence they must be equal.)
\item If $v(b)$ is not divisible by $p$, then $(a,b)$ is ramified and $\overline{(a,b)}$ is the field extension $\Fb(\sqrt[p]{\ba})$.  (By hypothesis, $a$ is not a $p$-th power in $\Fh$, so by Hensel's Lemma $\ba$ is not a $p$-th power in $\Fb$.)
\end{itemize}
\end{borel*}

\begin{prop} \label{crux.deg2}
Let $A_1$ and $A_2$ be degree $p$ symbol algebras over $F$ that are division algebras.  If
\begin{enumerate}
\item $A_1$ ramifies and $A_2$ does not; or
\item $A_1$ and $A_2$ both ramify but their ramification defines different extensions of the residue field $\Fb$,
\end{enumerate}
then there is a maximal subfield $L$ of one of the $A_i$ that does not split the other.
\end{prop}

\begin{proof}
Write $A_i$ as $(a_i, b_i)$.
We can assume $v(a_1) = v(a_2) = 0$.
If both the $A_i$ ramify, then the $v(b_i)$ must be prime to $p$ and 
the field extensions of the ramification of the $A_i$ must be 
$\Fb(\sqrt[p]{\bar a_i})$. If these are different fields,
then the field $L_2 := F(\sqrt[p]{a_2})$ cannot split $A_1$ because it does
not split the ramification of $A_1$.

Thus assume $A_2$ is unramified and $A_1$ ramifies.  In particular, $A_1 \ot \Fh$ is division, and we are done by Lemma \ref{krasner.2}.
\end{proof}

\begin{eg} \label{local.bad}
In the proposition, the hypothesis that $A_1$ and $A_2$ are both division is necessary.  Indeed, for $\ell \ne 2$, the field $\Q_\ell$ of $\ell$-adic numbers has two quaternion algebras: one split and one division.  Only one is ramified, but both are split by every quadratic extension of $\Q_\ell$. 
\end{eg}

\section{Discrete valuations: bad residue characteristic} \label{bad.char}

In this section, we consider quaternion algebras over a field with a dyadic discrete valuation.  
The purpose of this section is to prove:
 
\begin{prop}\label{char2}
Let $A_1, A_2$ be quaternion division algebras over a field $F$ of characteristic $0$ that has a dyadic discrete valuation.  If $A_1 \ot A_2$ is Brauer-equivalent to a ramified quaternion division algebra with residue algebra a separable quadratic extension of $\Fb$, then there is a quadratic extension of $F$ that embeds in $A_1$ or $A_2$ but not both.
\end{prop}

What really helps nail down the structure of $A_1 \ot A_2$ is the hypothesis that the residue algebra is a separable quadratic extension.  In particular, the division algebra underlying $A_1 \ot A_2$ is ``inertially split" in the language of \cite{JW}.

\begin{eg} \label{bad.char.eg}
Let $B$ be a quaternion division algebra over a field $F$ of characteristic 0 that is complete with respect to a dyadic discrete valuation, and suppose that the residue algebra $\Bb$ contains a separable quadratic extension $\Lb$ of $F$.  One can lift $\Lb/\Fb$ to a \emph{unique} quadratic extension $L/F$ and write $B = (L, b)$ for some $b \in \Fx$ that we may assume has value 0 or 1.  One finds two possibilities:
\begin{itemize}
\item If $v(b) = 0$, then $\Bb$ is a quaternion algebra over $\Fb$ generated by elements $i, j$ such that  
\[
\Fb(i) = \Lb, \ i^2 + i \in \Fb, \ j^2 = \bb, \eand ij = j(i + 1).
\]
In this case $B$ is unramified.
\item If $v(b) = 1$, then $\Bb = \Lb$ and $B$ is ramified.
\end{itemize}
For detailed explanation of these claims, see, e.g., \cite{GPe:comp}.
(Unlike the situation in the preceding section, there are more than just these two possibilities for general quaternion algebras over $F$, but in those other cases the residue algebra $\Bb$ is a purely inseparable extension of $\Fb$ of dimension 2 or 4 with $\Bb^2 \subseteq \Fb$.  Note that if $L/F$ is unramified and $\Lb/\Fb$ is purely inseparable, then $L/F$ is \emph{not} uniquely determined by $\Lb/\Fb$.) 
\end{eg}

\begin{proof}[Proof of Proposition \ref{char2}]
Put $B$ for the quaternion division $F$-algebra underlying $A_1 \ot A_2$.  By the hypothesis on its residue algebra, $B$ remains division when we extend scalars to the completion $\Fh$.  Hence at least one of $\Ah_1$, $\Ah_2$ is division; by Lemma \ref{krasner.2} (using that $A_1$ and $A_2$ are division) we may assume that both are division and that the valuation has the same ramification index on both algebras.  By Lemma \ref{krasner}, we may assume that $F$ is complete.

The residue division algebra $\Ab_1$ is distinct from $\Fb$ and we can find a quadratic subfield $L$ of $A_1$ so that the residue field $\Lb$ is a quadratic extension of $\Fb$.  
As $L$ is a maximal subfield of $A_1$, it is also one for $A_2$ and hence $B$.  Certainly, $\Lb$ is a subfield of the residue algebra $\Bb$; as both have dimension 2 over $\Fb$ they are equal.  That is, $\Lb$ is separable.  We write $A_i = (L, a_i)$ and $B = (L, b)$ for some $a_i, b \in \Fx$.  As the valuation does not ramify on $L$ we may assume that $v(b) = 1$.  By the previous paragraph, we may write $A_i = (L, a_i)$ for some $a_i \in \Fx$.  As $\Ab_i$ contains the separable extension $\Lb/\Fb$, we have $v(a_1) = v(a_2)$ as in Example \ref{bad.char.eg}.  Then $A_1 \ot A_2 \ot B$---which is split---is Brauer-equivalent to $(L, a_1 a_2 b)$, which is ramified (because $v(a_1 a_2 b) = 1 + 2v(a_i)$ is odd) and in particular not split.  This is a contradiction, which completes the proof of the proposition.
\end{proof}

\section{Unramified cohomology} \label{ur.sec}

For a field $F$ of characteristic not 2, we write $H^d(F)$ for the Galois cohomology group $H^d(F, \Zm2)$.  We remind the reader that for $d = 1$, this group is $\Fx / F^{\times 2}$ and for $d = 2$ it is identified with the 2-torsion in the Brauer group of $F$.

\begin{defn}
Fix an integer $d \ge 2$.  We define $H^d_u(F)$ to be the subgroup of $H^d(F)$ consisting of classes that are unramified at every nondyadic discrete valuation of $F$ (in the usual sense, as in \cite[p.~19]{GMS}) and are killed by $H^d(F) \ra H^d(R)$ for every real closure $R$ of $F$.
\end{defn}

\begin{eg} \label{hdu.eg}
\begin{enumerate}
\item A local field $F$ (of characteristic $\ne 2$) has $H^2_u(F) = 0$ if and only if $F$ is non-dyadic.  Indeed, $F$ has no orderings, so it suffices to consider the unique discrete valuation on $F$.
\item A global field $F$ (of characteristic $\ne 2$) has $H^2_u(F) = 0$ if and only if $F$ has at most 1 dyadic place.  This is Hasse's local-global theorem for central simple algebras \cite[8.1.17]{NSW2}.
\item A real closed field $F$ has $H^d_u(F) = 0$ for all $d \ge 2$; this is trivial.
\end{enumerate}
\end{eg}

A local field $F$ (of characteristic $\ne 2$) has $H^d_u(F) = 0$ for $d \ge 3$.  This is trivial because such a field has no orderings and cohomological 2-dimension $2$.

\begin{lem} \label{cd}
If $F(\sqrt{-1})$ has cohomological $2$-dimension $< d$, then $H^d_u(F)$ is zero.
\end{lem}

\begin{proof}
By \cite[Satz 3]{Ar:prim}, an element $x \in H^d(F)$ is zero at every real closure if and only if $x \cdot (-1)^r$ is zero for some $r \ge 0$.  On the other hand, there is an exact sequence for every $n$ \cite[30.12(1)]{KMRT}:
\[
\begin{CD}
H^n(F(\sqrt{-1})) @>{\cores}>> H^n(F) @>{\cdot (-1)}>> H^{n+1}(F) @>{\res}>> H^{n+1}(F(\sqrt{-1})
\end{CD}
\]
For $n \ge d$, the two end terms are zero, so the cup product with $(-1)^r$ defines an isomorphism $H^d(F) \iso H^{d+r}(F)$ for all $r \ge 0$.
\end{proof}

We immediately obtain:

\begin{cor}
A global field $F$ of characteristic $\ne 2$ has $H^d_u(F) = 0$ for all $d \ge 3$. $\hfill\qed$
\end{cor}

\begin{cor} \label{real.trdeg}
If $F$ has transcendence degree $d$ over a real-closed field, then $H^{d+1}_u(F) = 0$.
\end{cor}

\begin{proof}
$F(\sqrt{-1})$ has transcendence degree $d$ over an algebraically closed field, hence has cohomological dimension $\le d$ \cite[\S{II.4.3}]{SeCG}.
\end{proof}

We define a discrete valuation of an extension $F/F_0$ to be a discrete valuation on $F$ that vanishes on $F_0$.

\begin{prop} \label{ext.prop}
For every extension $F/F_0$, the homomorphism $\res \!: H^d(F_0) \ra H^d(F)$ restricts to a homomorphism $H^d_u(F_0) \ra H^d_u(F)$.
If additionally $F/F_0$ is unirational and 
\begin{equation} \label{ext.hyp}
\parbox{4.25in}{every class in $H^d(F)$ that is unramified at every discrete valuation of $F/F_0$ comes from $H^d(F_0)$,}
\end{equation}
then the homomorphism $H^d_u(F_0) \ra H^d_u(F)$ is an isomorphism.
\end{prop}

\begin{proof}
Fix $y \in H^d_u(F_0)$.  Every ordering of $F$ restricts to an ordering on $F_0$; write $R$ and $R_0$ for the corresponding real closures.  By hypothesis, the image of $y$ in $H^d(R_0)$ is zero, hence it also has zero image under the composition $H^d(F_0) \ra H^d(R_0) \ra H^d(R)$.
It is easy to see that $\res(y)$ is unramified at every nondyadic discrete valuation of $F$, and the first claim follows.

We now prove the second sentence.  The natural map $H^d(F_0) \ra H^d(F)$ is injective by \cite[p.~28]{GMS}.
Fix $x \in H^d_u(F)$.  As it is unramified at every discrete valuation of $F/F_0$, it is the restriction of some element $x_0 \in H^d(F_0)$ by \eqref{ext.hyp}.  We will show that $x_0$ belongs to $H^d_u(F_0)$.

Fix an extension $E := F_0(x_1, x_2, \ldots, x_n)$ containing $F$.  Every ordering on $F_0$ extends to an ordering $v$ on $E$ (and hence also on $F$) by a recipe as in \cite[pp.~9--11]{Lam:surv}.  As the map $H^d((F_0)_v) \ra H^d(F_v)$ is an injection by Lemma \ref{rc.lem} below, we deduce that $x_0$ is killed by every real-closure of $F_0$. 

One extends a nondyadic discrete valuation $v$ on $F_0$ to $E$ by setting $v(x_i) = 0$ and $\bar{x}_i$ to be transcendental over $\Fb_0$ as in \cite[\S{VI.10.1}, Prop.~2]{Bou:ac}; then the residue field of $\Eb$ is $\Fb_0(x_1, x_2, \ldots, x_n)$ and the natural map $H^{d-1}(\Fb_0) \ra H^{d-1}(\Eb)$ is injective, hence so is $H^{d-1}(\Fb_0) \ra H^{d-1}(\Fb)$.  It follows that the image of $x_0$ in $H^{d-1}(\Fb_0)$ is zero, which completes the proof.
\end{proof}

Here is the promised lemma:

\begin{lem} \label{rc.lem}
If $R_1 \subseteq R_2$ are real-closed fields, then the natural map $H^d(R_1) \ra H^d(R_2)$ is an isomorphism for every $d \ge 0$.
\end{lem}

\begin{proof}
$H^d(R_i)$ is the usual group cohomology $H^d(\Zm2, \Zm2)$.  Obviously $H^0(R_i) = H^1(R_i) = \Zm2$, and 2-periodicity for the cohomology of finite cyclic groups shows that $H^d(R_i) = \Zm2$ for all $d \ge 0$, with nonzero element $(-1)^d$.
\end{proof}

\begin{eg} \label{retract.u}
Recall that an extension $F$ of $F_0$ is \emph{retract rational} if it is the field of fractions of an $F_0$-algebra domain $S$ and there are $F_0$-algebra homomorphisms $S \ra F_0[x_1, \ldots, x_n](1/s) \ra S$ whose composition is the identity on $S$, where $F_0[x_1, \ldots, x_n]$ is a polynomial ring and $s \in F_0[x_1, \ldots, x_n]$ is a nonzero element.  Obviously such an extension is unirational, and \eqref{ext.hyp} holds by \cite[11.8]{Salt:lect} or \cite[Prop.~2.15]{M:urel}, so $H^d_u(F_0) \cong H^d_u(F)$.
\end{eg}

\section{Transparent fields} \label{transparent.sec}

\begin{defn} \label{ur.def}
Fix a field $F$ of characteristic $\ne 2$.  For each dyadic valuation $v$, fix a completion $\Fh$ of $F$.  Write $K$ for the limit of the finite unramified extensions of $\Fh$ with separable residue field.  This extension is called the \emph{maximal unramified extension of $\Fh$} (which highlights the fact that usually the word ``unramified" is quite a bit stronger than how we are using it).  There is a well-defined ramification map 
\[
H^2(K/\Fh, \mu_2) \ra H^1(\Fb, \Zm2) \quad \text{s.t.} \quad (a, b) \mapsto \begin{cases}
0 & \text{if $v(a) = v(b) = 0$}\\
(\overline{a}) & \text{if $v(a) = 0$ \& $v(b) = 1$}
\end{cases}
\]
We say that $F$ is \emph{transparent} if for every division algebra $D \ne F$ of exponent 2 such that $[D]$ belongs to $H^2_u(F)$, there is some dyadic place $v$ of $F$ and completion $\Fh$ such that 
\begin{enumerate}
\item $D \ot \Fh$ is split by the maximal unramified extension of $\Fh$, and
\item the image of $[D]$ in $H^1(\Fb, \Zm2)$ is not zero.
\end{enumerate}
Another way to phrase these hypotheses is: \emph{$D \ot \Fh$ is division and the residue division algebra $\bar{D}$ is a separable quadratic extension of $\Fb$.}

Roughly speaking, transparent fields are those for which nonzero 2-torsion elements of the Brauer group can be detected using ramification.
\end{defn}

\begin{eg}\begin{enumerate}
\setcounter{enumi}{-1}
\item If $H^2_u(F) = 0$---e.g., if $F$ is real closed---then $F$ is (vacuously, trivially) transparent.
\item \emph{If $F$ is local, then $F$ is transparent.}  If $F$ is nondyadic, this is Example \ref{hdu.eg}(1).  Otherwise, for each nonzero $x \in H^2_u(F)$, the residue division algebra is necessarily a separable quadratic extension of $\Fb$, because $\Fb$ is perfect with zero Brauer group.
\item If $F$ is global, then $F$ is transparent.  Indeed, $H^2_u(F)$ consists of classes ramified only at dyadic places, and every such place has finite residue field.
\end{enumerate}
\end{eg}

The \emph{unramified Brauer group} of an extension $F/F_0$ is the subgroup of the Brauer group of $F$ consisting of elements that are unramified at every discrete valuation of $F/F_0$.

\begin{prop} 
If $F_0$ is transparent, $F/F_0$ is unirational, and the unramified Brauer group of $F/F_0$ is zero, then $F$ is transparent.
\end{prop}

\begin{proof}
Suppose there is a nonzero $x \in H^2_u(F)$.  By Proposition \ref{ext.prop}, $x$ is the image of some nonzero $x_0 \in H^2_u(F_0)$ and there is a dyadic valuation $v$ of $F_0$ such that conditions (1) and (2) of Definition \ref{ur.def} holds for the division algebra $D_0$ represented by $x_0$.  Extend $v$ to a valuation on $F$ as in the proof of Proposition \ref{ext.prop}.  The maximal unramified extension of $\Fh$ (with respect to $v$) contains the maximal unramified extension of $\Fh_0$, so it kills $x_0$ and hence also $x$; this verifies condition (\ref{ur.def}.1).  For condition (\ref{ur.def}.2), it suffices to note that the map $H^1(\Fb_0, \Zm2) \ra H^1(\Fb, \Zm2)$ is injective as in the proof of Proposition \ref{ext.prop}.
\end{proof}

\begin{eg} \label{retract.transparent}
A retract rational extension of a transparent field is transparent, cf.~Example \ref{retract.u}.
\end{eg}

\section{Proof of the main theorem}

We now prove Theorem \ref{MT}.  We assume that $D_1$ and $D_2$ are not isomorphic and that there is some quadratic extension of $F$ that is contained in both of them, hence that $D_1 \ot D_2$ has index 2.  We will produce a quadratic extension of $F$ that is contained in one and not the other.

Suppose first that there is a real closure $R$ of $F$ that does not split $D_1 \ot D_2$.  Then one of the $D_i$ is split by $R$ and the other is not; say $D_1$ is split.  We can write $D_1 = (a_1, b_1)$ where $a_1$ is positive in $R$.
The field $F(\sqrt{a_1})$ is contained in $D_1$, but not contained in $D_2$ because $D_2 \ot R$ is division.

If there is a nondyadic discrete valuation of $F$ where $D_1 \ot D_2$ is ramified, then Proposition \ref{crux.deg2} provides the desired quadratic extension.

Finally, suppose that $D_1 \ot D_2$ is split by every real closure and is unramified at every nondyadic place, so $D_1 \ot D_2$ belongs to $H^2_u(F)$.  By the original hypotheses on $D_1$ and $D_2$, the class $[D_1 \ot D_2]$ is represented by a quaternion division algebra $B$.  But $F$ is transparent, so Proposition \ref{char2} provides a quadratic extension of $F$ that embeds in $D_1$ or $D_2$ but not the other. $\hfill\qed$

\section{Pfister forms and nondyadic valuations} \label{crux.sec}

With the theorem from the introduction proved, we now set to proving a version of it for Pfister forms.  The goal of this section is the following proposition, which is an analogue of Proposition \ref{crux.deg2}.

\begin{prop} \label{crux}
Let $F$ be a field with a nondyadic discrete valuation.  Let $\phi_1, \phi_2$ be $d$-Pfister forms over $F$ for some $d \ge 2$.  If 
\begin{enumerate}
\item $\phi_1$ is ramified and $\phi_2$ is not; or
\item $\phi_1$ and $\phi_2$ both ramify, but with different ramification,
\end{enumerate}
then there is a $(d-1)$-Pfister form $\gamma$ over $F$ such that $\gamma$ divides $\phi_1$ or $\phi_2$ but not both.
\end{prop}

The proof given below for case (1) of the proposition was suggested to us by Adrian Wadsworth, and is much simpler than our original proof for that case.  It makes use of the following lemma:

\begin{lem}\label{adrian}
Let $\psi = \qform{a_1, \ldots, a_n}$ be a nondegenerate quadratic form over a field $F$ with a nondyadic discrete valuation $v$.  If $v(a_i) = 0$ for all $i$ and $\psi$ is isotropic over the completion $\Fh$ of $F$, then there is a $2$-dimensional nondegenerate subform $q$ of $\psi$ such that $q \ot \Fh$ is isotropic.
\end{lem}

\begin{proof}
As $\phi \ot \Fh$ is isotropic, there exist $x_i \in \Fh$ not all zero so that $\sum a_i x_i^2 = 0$.  Scaling by a uniformizer, we may assume that $v(x_i) \ge 0$ for all $i$ and that at least one $x_i$ has value 0.  Let $j$ be the smallest index such that $v(x_j) = 0$; there must be at least two such indices, so $j < n$.  

Fix $t_i$ in the valuation ring of $F$ so that $\bar{t}_i = \bar{x}_i$, and put
\[
r := \sum_{i=1}^j a_i t_i^2 \eand s := \sum_{i = j+1}^n a_i t_i^2.
\]
By construction of $j$, $v(r) = 0$.  Since $r + s = 0$, also $v(s) = 0$.  Now $\qform{r, s}$ is a subform of $\psi$ (over $F$), and
\[
\bar{r} + \bar{s} = \overline{r+s} = \overline{\sum a_i t_i^2} = \overline{\sum a_i x_i^2} = 0.
\]
So $\qform{r,s}$ is isotropic over $\Fh$ \cite{Lam}.
\end{proof}

\begin{proof}[Proof of Proposition \ref{crux}]
Obviously, $\pform{ x\pi, y\pi} \cong \pform{-xy, y\pi}$, so we may assume that 
\[
\phi_i = \pform{a_{i2}, a_{i3}, \ldots, a_{id}, b_i}
\]
where $a_{ij}$ has value 0 and $b_i$ has value 0 or 1.

In case (2), $b_1$ and $b_2$ have value 1 and the residue forms 
\[
\ram( \phi_i) = \pform{ \ba_{i2}, \ba_{i3}, \ldots, \ba_{id}}
\]
are not isomorphic. We take $\gamma$ to be $\pform{a_{12}, a_{13}, \ldots, a_{1d}}$; it divides $\phi_1$.  The projective quadric $X$ defined by $\gamma = 0$ is defined over the discrete valuation ring, and $\Fb(X)$ does not split $\ram(\phi_2)$ by \cite[X.4.10]{Lam}.
Hence $F(X)$ does not split $\phi_2$ and $\gamma$ does not divide $\phi_2$.

Now suppose we are in case (1); obviously $v(b_1) = 1$.
If $\phi_2 \ot \Fh$ is anisotropic, then as $\phi_2$ does not ramify, $v(b_2) = 0$ and the second residue form of $\phi_2$ is zero.  Certainly $\gamma := \pform{a_{13}, \ldots, a_{1d}, b_1}$ divides $\phi_1$.  On the other hand, every anisotropic multiple of $\gamma \ot \Fh$ over $\Fh$ will have a nonzero second residue form, so $\phi_2$ cannot contain $\gamma$.
 
Otherwise, $\phi_2 \ot \Fh$ is hyperbolic, and there exists some 2-dimensional subform $q$ of $\phi_2$ that becomes hyperbolic over $\Fh$ by Lemma \ref{adrian}.  It follows 
that there is a $(d-1)$-Pfister $\psi$ that is contained in $\phi_2$ and that becomes hyperbolic over $\Fh$.  But $\phi_1 \ot \Fh$ is ramified, hence anisotropic, so $\phi_1$ cannot contain $\psi$. 
\end{proof}

\section{Theorem for quadratic forms} 

\begin{thm} \label{pf.thm}
Let $F$ be a field of characteristic $\ne 2$ such that $H^d_u(F) = 0$ for some $d \ge 2$.  Let $\phi_1$ and $\phi_2$ be anisotropic $d$-Pfister forms over $F$ such that, for every $(d-1)$-Pfister form $\gamma$, we have: $\gamma$ divides $\phi_1$ if and only if $\gamma$ divides $\phi_2$.  Then $\phi_1$ is isomorphic to $\phi_2$.
\end{thm}

We remark that for  $d = 2$, Theorem \ref{pf.thm} does not include the case where $F$ is a retract rational extension of a global field with more than one dyadic place, a case included in Theorem \ref{MT}.

\begin{proof}[Proof of Theorem \ref{pf.thm}]
Suppose that $\phi_1$ and $\phi_2$ are $d$-Pfister forms over $F$ for some $d \ge 2$, $H^d_u(F)$ is zero, and $\phi_1$ is not isomorphic to $\phi_2$.  Then $\phi_1 + \phi_2$ is in $H^d(F) \setminus H^d_u(F)$.  

If there is a nondyadic discrete valuation where $\phi_1 + \phi_2$ ramifies, then Proposition \ref{crux} gives a $(d-1)$-Pfister form that divides one of the $\phi_i$ and not the other. 
Otherwise,  there is an ordering $v$ of $F$ such that $\phi_1$ and $\phi_2$ are not isomorphic over the real closure $F_v$, i.e., one is locally hyperbolic and the other is not.  Write $\phi_i = \pform{a_{i1}, a_{i2}, \ldots, a_{id}}$ and suppose that $\phi_1$ is locally hyperbolic, so some $a_{1j}$ is positive; renumbering if necessary, we may assume that it is $a_{11}$.  Then the form $\pform{a_{12},  a_{13}, \ldots, a_{1d}}$ divides $\phi_1$, but it is hyperbolic over $F_v$ and so does not divide $\phi_2$.  
\end{proof}

\section{Appendix: tractable fields}

We now use the notion of unramified cohomology from \S\ref{ur.sec} to prove some new results and give new proofs of some known results regarding tractable fields.
Recall from \cite{CTW} that a field $F$ of characteristic $\ne 2$ is called \emph{tractable} if for every $a_1$, $a_2$, $a_3$, $b_1$, $b_2$, $b_3 \in \Fx$ such that
the six quaternion algebras $(a_i, b_j)$ are split for all $i \ne j$ and the three quaternion algebras $(a_i, b_i)$ are pairwise isomorphic, the algebras $(a_i, b_i)$ are necessarily split.  (This condition was motivated by studying conditions for decomposability of central simple algebras of degree 8 and exponent 2.)  
We prove:

\begin{prop} \label{tractable.prop}
If $H^2_u(F) = 0$, then $F$ is tractable.
\end{prop}

\begin{proof}
For sake of contradiction, suppose we are given $a_i, b_i$ as in the definition of tractable where the quaternion algebras $(a_i, b_i)$ are not split.  
If the $(a_i, b_i)$ do not split at some real closure of $F$, then the $a_i$ and $b_i$ are all negative there.  Hence $(a_i, b_j)$ for $i \ne j$ is also not split, a contradiction.

So the $(a_i, b_i)$ ramify at some nondyadic discrete valuation $v$ of $F$ and in particular are division over a the completion of $F$ at $v$.  We replace $F$ with its completion and derive a contradiction.  Modifying each $a_i$ or $b_i$ by a square if necessary, we may assume that each one has value 0 or 1.  Further, as each $(a_i, b_i)$ is ramified, at most one of the two slots has value 0.

Suppose first that one of the $a_i$ or $b_i$---say, $a_1$---has value 0.  Then $b_1$ has value 1 and $a_1$ is not a square in $F$.
As $(a_1, b_j)$ is split for $j \ne 1$, we deduce that $v(b_j) = 0$.   Hence $v(a_2) = v(a_3) = 1$.  As $(a_3, b_3)$ is division,  $b_3$ is not a square in $F$.   Since $v(a_2) = 1$ and $v(b_3) = 0$, it follows that $(a_2, b_3)$ is division, a contradiction.

It remains to consider the case where the $a_i$'s and $b_i$'s all have value 1.  We can write $a_i = \pi \alpha_i$ and $b_i = \pi \beta_i$, where $\pi$ is a prime element---say, $a_1$ so $\alpha_1 = 1$---and $\alpha_i, \beta_i$ have value 0.  In the Brauer group of $F$, we have for $i \ne j$:
\[
0 = (\pi \alpha_i, \pi \beta_j) = (\pi \alpha_i, -\alpha_i \beta_j) = (\pi, -\alpha_i \beta_j) + (\alpha_i, \beta_j).
\]
It follows that $-\alpha_i \beta_j$ is a square in $F$ for $i \ne j$.  In particular, $-\beta_2$, $-\beta_3$, $\alpha_2$, and $\alpha_3$ are all squares.  Then
\[
(a_2, b_2) = (\pi \alpha_2, \pi \beta_2) = (\pi, -\pi) = 0,
\]
a contradiction.
\end{proof}

\begin{eg}
In \cite{CTW}, the authors proved that a local field is tractable if and only if it is nondyadic (ibid., Cor.~2.3) and a global field is tractable if and only if it has at most 1 dyadic place (ibid., Th.~2.10).  Combining the preceding proposition and Example \ref{hdu.eg} gives a proof of the ``if" direction for both of these statements.
\end{eg}

\begin{cor}[{\protect{\cite[Th.~3.17]{CTW}}}]
Every field of transcendence degree 1 over a real-closed is tractable.
\end{cor}

\begin{proof}
Combine the proposition and Corollary \ref{real.trdeg}.
\end{proof}

\begin{eg}
The converse of Proposition \ref{tractable.prop} is false.  For every odd prime $p$, $\Q_p((x))$ is tractable by \cite[Prop.~2.5]{CTW}.  On the other hand, the usual discrete valuation on the formal power series ring is essentially the unique one \cite[4.4.1]{EnglerPrestel}, so $H^2_u(\Q_p((x))) = H^2(\Q_p) = \Zm2$.
\end{eg}

\begin{cor}
Let $F_0$ be a field of characteristic $\ne 2$ that
\begin{itemize}
\item is tractable and 
\item is global, local, real-closed, or has no $2$-torsion in its Brauer group.
\end{itemize}
Then every unirational extension $F/F_0$ satisfying \eqref{ext.hyp}---e.g., every retract rational extension $F/F_0$---is also tractable.
\end{cor}

\begin{proof}
We note that $F_0$ has $H^2_u(F_0) = 0$.  In case $F_0$ is global or local, then $F_0$ has one dyadic place or is non-dyadic respectively by \cite{CTW}, and the claim follows, see Example \ref{hdu.eg}.  Then $H^2_u(F)$ is also zero by Proposition \ref{ext.prop}, hence $F$ is tractable.
\end{proof}

Of course, this corollary includes the statement that rational extensions of tractable global fields are tractable, which was a nontrivial result (Example 3.14) in \cite{CTW}.  However, a stronger result has already been proved in \cite[Cor.~2.14]{Han:g0}: every rational extension of a tractable field is tractable.

We do get a large family of new examples of tractable fields:

\begin{cor}
Let $F$ be a global field with at most one dyadic place.  If $G$ is an isotropic, simply connected, absolutely almost simple linear algebraic group over $F$, then the function field $F(G)$ is tractable.
\end{cor}

\begin{proof}
The extension $F(G)/F$ is retract rational by Theorems 5.9 and 8.1 in \cite{Gille:KT}.
\end{proof}

\noindent{\small{\textbf{Acknowledgments.}  The first author's research was partially supported by National Science Foundation grant no.\ DMS-0653502.  Both authors thank Adrian Wadsworth and Kelly McKinnie for their comments on an earlier version of this paper.}}

\providecommand{\bysame}{\leavevmode\hbox to3em{\hrulefill}\thinspace}
\providecommand{\MR}{\relax\ifhmode\unskip\space\fi MR }
\providecommand{\MRhref}[2]{%
  \href{http://www.ams.org/mathscinet-getitem?mr=#1}{#2}
}
\providecommand{\href}[2]{#2}

\end{document}